\documentclass{amsart}
\usepackage{amssymb,amsmath}
\usepackage{ytableau}
\usepackage[enableskew]{youngtab}
\usepackage{young}
\usepackage{tikz,pgf}

\newtheorem{theorem}{Theorem}
\newtheorem{corollary}{Corollary}
\newtheorem{lemma}{Lemma}

\newtheorem{conjecture}{Conjecture}

\begin{document}

\title{Families of major index distributions: closed forms and unimodality}
\author{William J. Keith}

\begin{abstract}

Closed forms for $f_{\lambda,i} (q) := \sum_{\tau \in SYT(\lambda) : des(\tau) = i} q^{maj(\tau)}$, the distribution of the major index over standard Young tableaux of given shapes and specified number of descents, are established for a large collection of $\lambda$ and $i$.  Of particular interest is the family that gives a positive answer to a question of Sagan and collaborators. All formulas established in the paper are unimodal, most by a result of Kirillov and Reshetikhin.  Many can be identified as specializations of Schur functions via the Jacobi-Trudi identities.  If the number of arguments is sufficiently large, it is shown that any finite principal specialization of any Schur function $s_\lambda(1,q,q^2,\dots,q^{n-1})$ has a combinatorial realization as the distribution of the major index over a given set of tableaux.
\end{abstract}

\maketitle

\section{Introduction}

In this paper we establish formulas for the distribution of the major index of standard Young tableaux of fixed partition shape and number of descents, for numerous classes of each.  This work was initially motivated by a conjecture of Bruce Sagan and collaborators on the unimodality of the distribution of the major index over 321-avoiding permutations of length $n$ with fixed numbers of descents, which we also prove herein:

\begin{conjecture}\label{SagConj} \textbf{(Sagan et. al, \cite{SaganEtAl})} In the generating function $$\sum_{\sigma \in S_n(321)} q^{\text{maj} (\sigma)} t^{\text{des} (\sigma)} = \sum_{k=0}^{n-1} t^k A_{n,k}(q)$$ \noindent the polynomials $A_{n,k}(q)$ are unimodal.
\end{conjecture}

In a previous paper \cite{Keith1} the author established a formula for these polynomials, after summing with the Robinson-Schensted-Knuth correspondence:

\begin{theorem}\label{MainThm} The generating function for the major index of standard Young tableaux of shape $(n-k,k)$ with $i$ descents is \begin{align*}f_{(n-k,k),i}(q) &= \frac{q^{k+i^2-i}(1-q^{n-2k+1})}{1-q^i} \left[ { {k-1} \atop {i-1} } \right]_q \left[ { {n-k} \atop {i-1} } \right]_q \\ &= q^{i^2} \left( \left[ { {n-k} \atop {i} } \right]_q \left[ { {k} \atop {i} } \right]_q - \left[ { {n-k+1} \atop {i} } \right]_q \left[ { {k-1} \atop {i} } \right]_q  \right).\end{align*}
\end{theorem}

The second clause does not appear in that paper but its equivalence is a simple algebraic manipulation.  In that form it also appears in Barnabei et al. \cite{Barn2}, as an immediate consequence of Corollary 15 via the Robinson-Schensted-Knuth correspondence.

Establishing that this formula has a combinatorial interpretation in terms of a permutation statistic also establishes the positivity of the quotient in the first clause, and of the difference in the second.

Unimodality for this and the other families of polynomials discussed in this paper can be established in two ways: first as a result of Kirillov and Reshetikhin \cite{KR}, and second by identifying these families of polynomials as principal specializations of Schur functions.

Kirillov and Reshetikhin's result shows that the generating function of the major index of tableaux of a fixed shape and number of descents is a unimodal polynomial.  Sagan et al.'s conjecture immediately follows \emph{without} the formula -- the only additional work required is demonstrating that polynomials for all tableaux of shape $(n-k,k)$ with $n$ fixed and $k$ varying yield polynomials of the same central degree.

That families of Young tableaux of specified shape and number of descents should have such tidy formulas describing the distribution of their major index is of interest independent of Sagan et al.'s conjecture, and the main purpose of this paper is to establish several of these.  Many authors (\cite{AdinRoich}, \cite{Barnabei}, \cite{Barn2}, \cite{Butler}, \cite{EC2}, \cite{GoodOHaraStanton}, \cite{Kratt}) have done previous work on related classes of tableaux or related formulas.

The most important result of this type is Stanley's formula for the distribution of the major index of all standard Young tableaux of a given shape, without regard to number of descents \cite[Corollary 7.21.5]{EC2}: for $\lambda \vdash n$, 

\begin{equation}\label{StanFRT} \sum_{\tau \in SYT(\lambda)} = q^{\sum (i-1)\lambda_i)} \frac{(q)_n}{\prod (1-q^{h_{ij}})} \end{equation}

\noindent where the final product runs over all hooks in $\lambda$.  Adin and Roichmann \cite[Proposition 10.15]{AdinRoich} give a closed form for two-rowed tableaux, $$\sum_{\tau \in SYT((n-k,k))} q^{maj(\tau)} = \left[ {n \atop k} \right]_q - \left[ {n \atop {k-1}} \right]_q.$$ a formula also studied by Andrews \cite{Andrews}, and Reiner and Stanton \cite{ReinerStanton}.

These authors, prior to Adin and Roichmann's result, observed that the difference of $q$-binomial coefficients was not obviously nonnegative, but appeared to be positive and unimodal.  Andrews established positivity, and Reiner and Stanton established unimodality.

The theorems of this paper refine these theorems by number of descents, and produce further formulas, for selected families of shape $\lambda$.  Polynomiality and positivity of the formulas described, which are not always obviously nonnegative or polynomials, are thus given by the establishment of this combinatorial description, and unimodality by invoking the result of Kirillov and Reshetikhin.

In many of these cases, we can establish unimodality by a second means of independent interest, identifying these distributions as a power of $q$ times the principal specializations of particular Schur polynomials.  This is done by showing the polynomials to be particular instances of the Jacobi-Trudi identity.  Because it is known that a principal specialization of a Schur polynomial is unimodal, this establishes unimodality independently of the Kirillov/Reshetikhin result.

The main theorems of this paper are as follows.  In all cases $$f_{\lambda \setminus \mu,i} := \sum_{{\tau \in SYT(\lambda \setminus \mu)} \atop {des(\tau) = i}} q^{maj(\tau)}.$$

Established in an earlier paper of the author's \cite{Keith1} were the following formulas, the first of which we restate for completeness:

\begin{theorem}\label{PrevRes} 
\begin{align*}
f_{(n,k),i} &= \frac{q^{k+i^2-i}(1-q^{n-k+1})}{1-q^i}\left[ { {n} \atop {i-1} } \right]_q \left[ { {k-1} \atop {i-1} } \right]_q \\
 &= q^{i^2} \left( \left[ { {n} \atop {i} } \right]_q \left[ { {k} \atop {i} } \right]_q - \left[ { {n+1} \atop {i} } \right]_q \left[ { {k-1} \atop {i} } \right]_q  \right) \\
f_{(n,k,1),i} &= q^{k+i^2-2i+2} \frac{(1-q^{n-k+1})(1-q^{i-1})}{(1-q^i)(1-q)}\left[{k \atop {i-1}} \right]_q \left[ {{n+1} \atop {i-1}} \right]_q
\end{align*}
\end{theorem}

It does not appear that $f_{(n,k,1),i}$ is another instance of a Jacobi-Trudi determinant, so the Kirillov/Reshetikhin result is the only route at present by which we can establish the unimodality of the product described above.

Presented at the Joint Meetings of the AMS and the MAA in 2018 were the two additional formulas below. The first, Theorem \ref{SkewThm}, generalizes Theorem \ref{PrevRes} to skew tableaux; the version presented here is corrected.  The full proofs are given herein.

\begin{theorem}\label{SkewThm} For $n \geq k > 0$, $j < n$, $i \geq 1$,
$$f_{(n,k) \setminus (j),i}(q) = q^{i^2} \left( \left[ { {n-j} \atop {i} } \right]_q \left[ { {k} \atop {i} } \right]_q - \left[ { {n+1} \atop {i} } \right]_q \left[ { {k-j-1} \atop {i} } \right]_q \right).$$
\end{theorem}

\begin{theorem}\label{MK2Thm} For $n \geq k \geq 2$, $i \geq 2$, 
$$f_{(n,k,2),i}(q) = q^{k+i^2-3i+6} \frac{(1-q^{n-k+1})(1-q^{k-1})(1-q^n)}{(1-q^{i-1})(1-q)(1-q^2)} \left[ { {n+1} \atop {i-2} } \right]_q \left[ { {k} \atop {i-2} } \right]_q.$$
\end{theorem}

Again calculation suggests that Theorem \ref{MK2Thm} is not an instance of a Schur function specialization, but Kirillov and Reshetikhin suffices to guarantee unimodality.

Theorem \ref{SkewThm} is not covered by the result of Kirillov and Reshetikhin, which applies to standard tableaux of partition shape; however, as a consequence of the concentricity of the formulas in Theorem \ref{PrevRes} necessary to show Sagan et al.'s conjecture, we will be able to establish the unimodality of this formula, yielding the following new combinatorial theorem.

\begin{theorem}\label{SkewMaj} The distribution of the major index over all skew two-rowed standard Young tableaux with fixed number of descents is a unimodal polynomial.
\end{theorem}

Presented at the conference Combinatory Analysis 2018 was the case $\lambda = (\lambda_1, \lambda_2, \lambda_3)$ of the following general theorem.

\begin{theorem}\label{BigSchur} Given $\lambda = (\lambda_1, \dots , \lambda_r) \vdash n$, let $\alpha = \alpha(\lambda) = (\lambda_2,\dots,\lambda_r)$.  Then $$f_{\lambda, n-\lambda_1} = q^{\binom{n-\lambda_1+1}{2}}s_{\alpha^\prime}(1,q,\dots,q^{\lambda_1-1}).$$
\end{theorem}

This has an interesting corollary giving a combinatorial interpretation for any Schur function's principal specialization, as long as the number of $q$-power arguments is sufficiently large.

\begin{corollary} For any partition $\beta$, if $k \geq {\beta^\prime}_1$, that is, if $k$ is at least the number of parts in $\beta$, then the principal specialization $s_\beta(1,q,\dots,q^{k-1})$ of the Schur function indexed by $\beta$ is, up to shift by a power of $q$, the distribution of the major index of all standard Young tableaux of shape $\lambda = (k, {\beta^\prime}_1,{\beta^\prime}_2,\dots)$ with the maximum possible number of descents.
\end{corollary}

In Section \ref{DefSec} we give all the definitions necessary for the paper and the known results we will need.  In Section \ref{FormSec} we confirm Sagan et al.'s conjecture, and prove the theorems listed above, as well as a few further families.  In Section \ref{FutureSec} we discuss possible future directions of the project and the challenges associated with further progress.

\section{Definitions and Background}\label{DefSec}

A \emph{partition} of $n$ is a nonincreasing sequence $\lambda = (\lambda_1,\lambda_2,\dots)$ of nonnegative integers that sums to $n$.  We denote $\vert \lambda \vert = n$ or $\lambda \vdash n$.  The $\lambda_i$ are the parts of $\lambda$.  It will be convenient in this paper to regard partitions as being infinite sequences, of which necessarily only finitely many entries are nonzero.  Occasionally we use the notation $\lambda = (a_1^{b_1},a_2^{b_2},\dots)$ to mean the partition with $b_1$ parts of size $a_1$, $b_2$ parts of size $a_2$, etc.  

A \emph{composition} is a sequence of nonnegative integers that sums to $n$; it need not be nonincreasing.  If for two partitions $\lambda$ and $\mu$ we have $\mu_i \leq \lambda_i$ for all $i$, then the \emph{skew partition} $\lambda \setminus \mu$ partitions $\vert \lambda \vert - \vert \mu \vert$, and is best described with its Ferrers diagram.

The \emph{Ferrers diagram} of $\lambda = (\lambda_1,\dots) \vdash n$ is an array of $n$ unit squares in the fourth quadrant, justified to the origin, wherein a box exists with bottom right corner $(-i,-j)$ if $\lambda_ i \geq j$.  The \emph{conjugate} of a partition $\lambda$, denoted $\lambda^{\prime} = (\lambda_1^{\prime},\dots)$, is the partition with Ferrers diagram that of $\lambda$, reflected across the line $y=-x$ in the plane.  The \emph{hooklength} of the box with bottom right corner at $(-i,-j)$, denoted $h_{ij}$, is $\lambda_j - i + \lambda_i^{\prime} - j + 1$; it equals the number of boxes directly right of or below the box at $(-i,-j)$, plus 1 for itself.  The Ferrers diagram of the skew partition $\lambda \setminus \mu$ is the set of boxes in the Ferrers diagram of $\lambda$ but not in the Ferrers diagram of $\mu$.

\begin{figure}[h]
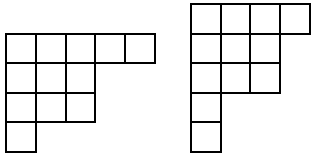

$\young(\hfil\hfil\hfil\hfil\hfil,\hfil\hfil\hfil::,\hfil\hfil\hfil::,\hfil::::) \, \quad \, \young(\hfil\hfil\hfil\hfil,\hfil\hfil\hfil:,\hfil\hfil\hfil:,\hfil:::,\hfil:::)$
\caption{The Ferrers diagrams of, left, $\lambda = (5,3,3,1)$ and, right, its conjugate $(4,3,3,1,1)$.}
\end{figure}

\begin{figure}[h]
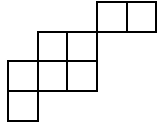

$\young(:::\hfil\hfil,:\hfil\hfil::,\hfil\hfil\hfil::,\hfil::::)$
\caption{The Ferrers diagram of $\lambda \setminus \mu = (5,3,3,1) \setminus (3,1)$.}
\end{figure}

A \emph{standard Young tableau} of shape $\lambda \vdash n$ is a filling of the Ferrers diagram of $\lambda$ with the numbers 1 through $n$ such that rows increase left to right and columns increase top to bottom.  The set of these is denoted $SYT(\lambda)$.  If the entry $i$ is in a higher row than the entry $i+1$, we say that the tableau has a \emph{descent} at place $i$; we say $i$ is the descent top and $i+1$ is the descent bottom.  The set of descents of a tableau $\tau$ is $Des(\tau)$.  Tableaux possess several useful combinatorial statistics, among them the \emph{descent number} $des(\tau) = \vert Des(\tau) \vert$ and the \emph{major index} $maj(\tau) = \sum_{i \in Des(\tau)} i$.

\begin{figure}[h]
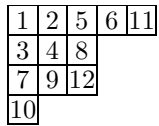

$\young(1256{{11}},348,79{{12}},{{10}})$
\caption{An element $\tau$ of $SYT((5,3,3,1))$ with $des(\tau) = 5$, $maj(\tau) = 36$.}
\end{figure}

The \emph{skew standard Young tableaux} are defined similarly, and the descent and major index statistics remain well-defined.  Note that even if 1 is not in the top row of a skew tableau, it does not constitute the bottom of a descent.

A \emph{semistandard Young tableaux} may repeat entries.  Entries must strictly increase from top to bottom down columns, but need only weakly increase from left to right.  The set of these of shape $\lambda$ is denoted $SSYT(\lambda)$.

\begin{figure}[h]
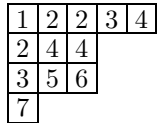

$\young(1223{{4}},244::,35{{6}}::,{{7}}:::)$
\caption{An element $\tau$ of $SSYT((5,3,3,1))$.}
\end{figure}

The number of standard Young tableaux of shape $\lambda$ is given by the Frame-Robinson-Thrall hooklength formula $$g^{\lambda} = \frac{n!}{\prod h_{ij}}$$ where the product runs over all boxes in the diagram.

The Stanley formula of equation \ref{StanFRT} for non-skew tableaux is thus a $q$-analogue of this number.  In part of a very recent paper \cite[Section 9]{MPP}, Morales, Pak and Panova survey the known counting formulas for skew tableaux.  The present Theorem \ref{SkewMaj} is a $q$-counting in the two-row case for skew tableaux.

A polynomial $p(q) = p_0 + p_1 q + \dots + p_d q^d$ is \emph{unimodal} if there is some $a$ such that $$p_0 \leq p_1 \leq \dots \leq p_a \geq p_{a+1} \geq \dots \geq p_d.$$  It is \emph{symmetric} if, when $j$ is the minimal degree such that $p_j \neq 0$, it holds that $p_j = p_d$, $p_{j+1} = p_{d-1}$, etc.  In such a case the polynomial has \emph{central degree} $(j+d)/2$, or Zeilberger prefers the term \emph{darga} of $j+d$.

When the coefficients of such a polynomial are nonnegative integers, a classical method to prove unimodality is to construct partially ordered sets of combinatorial objects in which the populations of the ranks are counted by the $p_i$, and construct an injection from less-populated to more-populated ranks.  The coefficient of $q^n$ in the $q$-binomial $\left[ { {M+N} \atop {N}} \right]_q$, defined by $$(q)_m = (1-q)(1-q^2)\dots(1-q^m) \quad , \quad \left[ {{M+N} \atop {N}} \right]_q = \frac{(q)_{M+N}}{(q)_M(q)_N}, $$ \noindent counts the number of partitions of $n$ with at most $M$ positive parts, each of size at most $N$, or partitions in the $M \times N$ box.  When $M < N$ we take the convention that the value of $\left[ { {M} \atop {N} } \right]_q$ is 0.  If $M \geq 0$, then $\left[ { {M} \atop {0} } \right]_q = 1$.  

The unimodality of the $q$-binomial coefficient was proven first by P. A. MacMahon with algebraic methods and later in a famous result by constructing the combinatorial injection by Kathleen O'Hara \cite{OHara}.  O'Hara's method was to show that the $q$-binomial coefficient could be expressed as a sum of smaller $q$-binomial coefficients.  Denote by $\lambda \vdash n$ that $\lambda = (\lambda_1,\lambda_2,\dots)$ is a partition of $n$, and let $Y_i = \sum_{j=1}^i \lambda_j$, $Y_0 = 0$.  Then her structure theorem can be rendered, following Zeilberger \cite{Zeil}, as

\begin{theorem} \textbf{The KOH Theorem:} $$\left[ {{n+a} \atop {n}} \right]_q = \sum_{\lambda \vdash n} q^{\sum_{i=1}^\infty \lambda_i^2 - \lambda_i}\prod_{j \geq 1} \left[ {{j(a+2) -Y_{j-1}-Y_{j+1}} \atop {\lambda_j - \lambda_{j-1}}} \right]_q.$$
\end{theorem}

We additionally observe that a product of unimodal symmetric polynomials is unimodal and symmetric, and that a sum of unimodal symmetric polynomials with nonnegative coefficients having the same central degree is itself symmetric and unimodal with nonnegative coefficients and the same central degree.  With a short calculation showing that every summand polynomial has the same central degree, and finally the base case that $\left[ {{M+1} \atop 1} \right]_q = 1+q +\dots + q^M$, the unimodality of the $q$-binomial coefficient was inductively proved by O'Hara's structure theorem above.

A \emph{permutation} of length $n$ is an ordered list of the numbers 1 through $n$.  If in a permutation $\sigma$ the element $i$ is followed immediately by element $j$ with $i>j$, then $\sigma$ has a \emph{descent} at place $i$.  The set of descents of the permutation is $Des(\sigma)$.  The \emph{descent number} $des(\sigma)$ of the permutation is the cardinality of this set, and the \emph{major index} $maj(\sigma)$ of the permutation is its sum.

The Robinson-Schensted-Knuth correspondence (hereinafter \emph{RSK}) gives a bijection between permutations of length $n$ and ordered pairs of standard Young tableaux of the same shape which are partitions of $n$.  This bijection has the properties that the descent set of the permutation is equal to the descent set of the second member of the tableaux pair, and that the number of rows in the image partition shapes is equal to the length of the longest decreasing (not necessarily contiguous) subsequence of the permutation.

If a permutation of length $n$ has no descending subsequence of length 3 or more, we say that the permutation is 321-\emph{avoiding}.  By the remarks above, the partition pairs corresponding to this permutation under RSK are of shapes having 1 or 2 parts.  Denote the set of such permutations by $S_n(321)$.

\subsection{The Kirillov-Reshetikhin theorem}

A \emph{semistandard Young tableau} of shape $\lambda \setminus \rho$ and content $\mu$, with $\mu = (\mu_1,\dots)$ a composition of $\vert \lambda \setminus \rho \vert$, is a filling of the Ferrers diagram of $\lambda \setminus \rho$ with $\mu_1$ ones, $\mu_2$ twos, etc., with rows nondecreasing left to right and columns strictly increasing top to bottom.  The set of semistandard Young tableaux of shape $\lambda \setminus \rho$ and content $\mu$ is denoted $SSYT(\lambda \setminus \rho,\mu)$.  The set of standard Young tableaux of shape $\lambda$ is then just $SYT(\lambda) = SSYT(\lambda,1^{\vert \lambda \vert})$. If $\mu$ is a partition, then a semistandard Young tableau of shape $\lambda$ and content $\mu$ possesses a statistic known as the \emph{charge}, the definition of which we will not need in this paper.  The polynomial $$K_{\lambda\mu}(q) = \sum_{\tau \in SSYT(\lambda,\mu)} q^{charge(\tau)}$$ is known as the \emph{Kostka polynomial}.

One tool we will use to prove unimodality of formulas in this paper is Kirillov and Reshetikhin's result that the generating function for charge over standard tableaux with a fixed number of descents is unimodal, and is the same as the generating function for major index over the same set, up to a power shift.  The presentation given here follows Goodman, O'Hara, and Stanton \cite{GoodOHaraStanton}.  If referencing \cite{KR} directly note that there $\alpha^i$ is the conjugate of $\nu^{(i)}$.

Let the generating function for charge over standard Young tableaux of shape $\lambda$ with exactly $k$ descents be denoted $$K_{\lambda,1^{\vert \lambda \vert}}^k(q) = \sum_{{\tau \in SYT(\lambda)} \atop {\vert Des(\tau) \vert = k}} q^{charge(\tau)}.$$

Let the generating function for the major index over tableaux of shape $\lambda$ with exactly $k$ descents be denoted $$f_{\lambda}^k (q) = \sum_{{\tau \in SYT(\lambda)} \atop {\vert Des(\tau) \vert = k}} q^{maj(\tau)}.$$

These two polynomials are related by \cite[equation (1.2)]{GoodOHaraStanton} $$K_{\lambda,1^{\vert \lambda \vert}}^k (q) = q^{\binom{\vert \lambda \vert}{2}} f_{\lambda}^k(q^{-1}).$$

Because both are symmetric, this means that the two polynomials are simply shifts of the other by a power of $q$, and hence if one is unimodal then the other is.

Given two partitions $\lambda$ and $\mu$ with $\vert \lambda \vert = \vert \mu \vert$, let an admissible sequence $\alpha$ be a sequence of partitions of the form $\alpha = (\alpha^0, \alpha^1, \alpha^2, \dots )$ in which $\alpha^0 = \mu^\prime$, and for $i \geq 1$, $\vert \alpha^i \vert = \sum_{j=i+1}^\infty \lambda_j$.  For any such sequence $\alpha$, define the quantity $$c(\alpha) = \sum_{a,i \geq 1} \binom{\alpha^{a-1}_i - {\alpha^a}_i}{2}.$$  For $a, i \geq 1$ and a given $\alpha$, define the function $$P_i^a (\alpha) = \sum_{j=1}^i ( {\alpha^{a-1}}_j - 2{\alpha^a}_j + {\alpha^{a+1}}_j ).$$  Then Kirillov and Reshetikhin give in \cite{KR} the following formula for $K_{\lambda, 1^{\vert \lambda \vert}}^k (q)$.  

\begin{theorem}\label{KRThm} \textbf{\cite[Theorems 4.2, 4.7 (iii)]{KR}} For a partition $\lambda$, $$K_{\lambda, 1^{\vert \lambda \vert}}^k (q) = \sum_{{\alpha = (\mu^\prime, \alpha_1,\alpha_2,\dots)} \atop {\alpha_1^1 = k}} q^{c(\alpha)} \prod_{k,i} \left[ { {P_i^a(\alpha) + {\alpha^a}_i - {\alpha^a}_{i+1}} \atop {{\alpha^a}_i - {\alpha^a}_{i+1}} } \right]_q.$$
\end{theorem}

Finally, \cite[equation (4.2)]{KR} establishes that in the case that interests us, where $\mu^\prime$ is the one-part partition $(\vert \lambda \vert)$ and all terms in the sum have $\alpha_1^1 = k$, we have that each summand is a (clearly symmetric, nonnegative, and unimodal) polynomial of the same central degree:  \begin{equation}\label{cendegree} 2c(\alpha) + \sum_{a,i \geq 1} P_i^a(\alpha)({\alpha^a}_i - {\alpha^a}_{i+1}) = 2 \binom{\vert \lambda \vert}{2} - \vert \lambda \vert k.\end{equation}

This establishes that $K_{\lambda,1^{\vert \lambda \vert}}^k (q)$ and hence $f_{\lambda}^k (q)$ are unimodal.

\subsection{Schur polynomials and principal specializations.} Among numerous equivalent expressions for the Schur polynomials, probably the simplest is the following.  Let variables $x_1,\dots,x_n$ be specified.  Given $\tau \in SSYT(\lambda \setminus \rho)$ in which $t_1$ ones appear, $t_2$ twos, $\dots$, denote $x^\tau = x_1^{t_1}x_2^{t_2}\dots x_n^{t_n}$.  Then the Schur polynomial indexed by $\lambda \setminus \rho$ is $$s_{\lambda \setminus \rho}(x_1,\dots,x_n) = \sum_{\tau \in SSYT(\lambda \setminus \rho)} x^\tau.$$

Our primary interest is in the fact that the \emph{principal specialization} $s_{\lambda} (1,q,\dots,q^n)$ is always a unimodal polynomial \cite{GoodOHaraStanton}.  Two determinantal expressions for Schur polynomials are the \emph{Jacobi-Trudi identities}, which we will employ in their specialized forms (\cite[Theoren 7.16.1 and 7.16.2]{EC2} and \cite[pp. 27 and 41]{Mac}): for $\lambda = (\lambda_1,\dots,\lambda_r)$, $\rho = (\rho_1,\dots,\rho_r)$, $0 \leq \rho_i \leq \lambda_i$,  

\begin{align*}
s_{\lambda \setminus \rho}(1,q,\dots,q^n) &= det \left( \left[ {{n-1+\lambda_i-\rho_j+i+j} \atop {n-1} } \right]_q \right)_{1\leq i,j \leq r} \\
 &= det \left( \left[ {{n} \atop {\lambda^\prime_i-\rho^\prime_j-i+j} } \right]_q q^{\binom{\lambda^\prime_i-\rho^\prime_j-i+j}{2}} \right)_{1\leq i,j \leq \lambda_1} \, .
\end{align*}

\section{Formulas}\label{FormSec}

\noindent \textbf{Proof of Conjecture \ref{SagConj}.} We begin by proving the motivating conjecture of Sagan et al.

By the RSK correspondence, we have that the generating function for the major index of all 321-avoiding permutations of length $n$ with exactly $i$ descents is \begin{equation}\label{ank} A_{n,i}(q) = \sum_{k=i}^{\lfloor \frac{n}{2} \rfloor} g^{(n-k,k)} f_{(n-k,k),i} (q) . \end{equation}

By either Kirillov and Reshetikhin or by interpretation of the second expression for $f_{(n-k,k),i}$ in Theorem \ref{MainThm} as a shift of the principal specialization of a Schur function, we know that these nonnegative polynomials are each unimodal.  By analyzing the first form of the generating function $f_{(n-k,k),i}$ given in Theorem \ref{MainThm}, we see that all of the summands in (\ref{ank}) have the same central degree $ni/2$, and hence their sum is also unimodal.

Conjecture \ref{SagConj} is thus established. \hfill $\Box$

\phantom{.}

We now consider further instances of this phenomenon.

\phantom{.}

\noindent \textbf{Skew tableaux; proof of Theorem \ref{SkewThm}.} We instead prove the identity $$f_{(n,k) \setminus (j),i}(q) = \sum_{r=1}^{j+1} q^{i^2 + i(r-1)} \left( \left[ { {n-j} \atop {i} } \right]_q \left[ { {k-r} \atop {i-1} } \right]_q - \left[ { {n-r+1} \atop {i-1} } \right]_q \left[ { {k-j-1} \atop {i} } \right]_q \right).$$

This is equivalent to the claim of the theorem after employing the $q$-binomial summation identity: 

\begin{lemma}\label{SkewSum}$$\sum_{R=0}^j q^{iR} \left[ { {A - R - 1} \atop {i-1} } \right]_q = \left[ { {A} \atop {i} } \right]_q - q^{i(j+1)} \left[ { {A - j - 1} \atop {i} } \right]_q .$$
\end{lemma}

\noindent \emph{Proof of Lemma \ref{SkewSum}.} Combinatorially interpret each term on the left-hand side as counting partitions in the $(A-i) \times i$ box in which exactly $R$ parts are of size exactly $i$; the sum then counts all partitions in the $(A-i) \times i$ box in which there are at most $j$ parts of size $i$. The difference on the right-hand side is precisely the count of partitions in the $(A-i) \times i$ box, less those partitions with more than $j$ parts of size $i$.  Thus the two sides are equal.\hfill $\Box$

\phantom{.}

We claim that the summand in the formula with index $r$ gives the distribution of major index for those tableaux with $i$ descents which have shape as illustrated below: the skew partition $(n,k) \setminus (j)$ in which $r$ is the first element in the top row, that is, entries 1 through $r-1$ begin the bottom row, and $r$ is in the top row.  This requires that $r$ be at minimum 1 and at maximum $j+1$.

$$\young(:::::::{r}*****,12{\dots}{r-1}*{*\,\,}****:::)$$

We begin with the $r=1$ case.  Let $f_{(n,k) \setminus (j), i}^*$ denote the distribution of the major index over skew standard Young tableaux of shape $(n,k) \setminus (j)$ with $i$ descents in which the entry 1 is in the top row. For this term the formula simplifies to \begin{equation}\label{fstar} f_{(n,k) \setminus (j), i}^* = q^{i^2} \left( \left[ { {n-j} \atop {i} } \right]_q \left[ { {k-1} \atop {i-1} } \right]_q - \left[ { {n} \atop {i-1} } \right]_q \left[ { {k-j-1} \atop {i} } \right]_q \right) .\end{equation}

We construct a recurrence for $f_{(n,k) \setminus (j), i}^*$, beginning with base case $i=1$.  In this case the first and only descent must follow the entry of value no less than $max(1,k-j)$ and no greater than $n-j$, inclusive.  Hence we have that either $n-j < 1$ or $k < 1$, in which case there are no such tableaux and the generating function is 0, or $$f_{(n,k) \setminus (j), 1}^* = q^{max(1,k-j)} + \dots + q^{n-j},$$ \noindent which in both cases is the claim of equation (\ref{fstar}).

Now let $i > 1$.  Note that this forces $n-j \geq i > 1$.  In a skew tableau of shape $(n,k) \setminus (j)$ with $i>1$ descents and 1 in the first row, the entry for box $n-j+k$ is either on the end of the first row, in which case it can be removed to leave any skew tableau of shape $(n-1,k) \setminus (j)$ with $i$ descents, or on the end of the second row following a sequence of $\ell+1$ consecutive boxes concluding with the last box of the first row containing value $n-j+k-\ell$. Removing these $\ell+1$ boxes results in a tableau of shape $(n-1,k-\ell) \setminus (j)$ with exactly $i-1$ descents.  This removal reduces the major index of this tableau by $n-j+k-\ell$.  Because the remaining tableau has at least one descent, the first entry of the top row has not been removed.  Thus for $i>1$ we have the recurrence

$$f_{(n,k) \setminus (j),i}^* = f_{(n-1,k) \setminus (j),i}^* + \sum_{\ell=1}^{k-1} q^{n-j+k-\ell} f_{(n-1,k-\ell) \setminus(j),i-1}^*. $$

We establish boundary conditions.  Observe that if $n=k$, then the term $f_{(n-1,k) \setminus (j),i}^*$ always yields 0 for any values of $j$ and $i$ since the two terms of the difference are equal, and this matches combinatorial requirement that $n \geq k$.  If $n \geq k$ but $n-j < i$ or $k<i$ then both terms yield 0, which is also correct.  Hence we may assume for induction that $n > k$ and that the formula holds true for $i-1$ descents and all values of the parameters, and for $i$ descents and smaller values of $n$ with any values of the other parameters.

We therefore substitute the claimed formula into the recurrence and sum, using the well-known $q$-binomial summation identities $\sum_{j=0}^n q^j \left[ { {m+j} \atop m} \right]_q = \left[ {{n+m+1} \atop {m+1} } \right]_q$ and $ \left[ { {A-1} \atop {B} } \right]_q +  q^{A-B} \left[ { {A-1} \atop {B-1} } \right]_q =  \left[ { {A} \atop {B} } \right]_q$:

\begin{align*}
f_{(n,k) \setminus (j),i}^* (q) &= q^{i^2} \left( \left[ { {n-1-j} \atop {i} } \right]_q \left[ { {k-1} \atop {i-1} } \right]_q - \left[ { {n-1} \atop {i-1} } \right]_q \left[ { {k-j-1} \atop {i} } \right]_q \right) \\
& + \sum_{\ell=1}^{k-1} q^{n-j+\ell} \left( q^{(i-1)^2} \left( \left[ { {n-1-j} \atop {i-1} } \right]_q \left[ { {\ell-1} \atop {i-2} } \right]_q - \left[ { {n-1} \atop {i-2} } \right]_q \left[ { {\ell-j-1} \atop {i-1} } \right]_q \right) \right) \\
&= q^{i^2} \left( \left[ { {n-1-j} \atop {i} } \right]_q \left[ { {k-1} \atop {i-1} } \right]_q - \left[ { {n-1} \atop {i-1} } \right]_q \left[ { {k-j-1} \atop {i} } \right]_q \right) \\
& + q^{n-j+i^2-i} \left[ { {n-1-j} \atop {i-1} } \right]_q \sum_{d=2-i}^{k-i} q^d  \left[ { {(i-2)+d} \atop {i-2} } \right]_q \\
& + q^{n+i^2-i+1}  \left[ { {n-1} \atop {i-2} } \right]_q \sum_{d=1-i-j}^{k-i-j+1} q^d  \left[ { {(i-1)+d} \atop {i-1} } \right]_q \\
&= q^{i^2}  \left[ { {k-1} \atop {i-1} } \right]_q \left( \left[ { {n-1-j} \atop {i} } \right]_q +  q^{n-j-i} \left[ { {n-1-j} \atop {i-1} } \right]_q \right) \\
& - q^{i^2}  \left[ { {k-j-1} \atop {i} } \right]_q \left(  \left[ { {n-1} \atop {i-1} } \right]_q + q^{n-i+1} \left[ { {n-1} \atop {i-2} } \right]_q \right) \\
&= q^{i^2} \left(  \left[ { {k-1} \atop {i-1} } \right]_q  \left[ { {n-j} \atop {i} } \right]_q -  \left[ { {k-j-1} \atop {i} } \right]_q  \left[ { {n} \atop {i-1} } \right]_q \right)
\end{align*}

This proves the formula for $f_{(n,k) \setminus (j),i}^*$.

To prove the theorem we observe that among those skew tableaux with exactly $i$ descents, where entries 1 through $r-1$ begin the second row, the resulting distribution of the major index is precisely that of those skew tableaux with $i$ descents, where 1 is in the first row, all of $n$, $k$, and $j$ are decreased by $r-1$, and the major index has $i(r-1)$ added to it.  Summing over valid $r$ yields the theorem. \hfill $\Box$

\phantom{.}

\noindent \textbf{Proof of Theorem \ref{SkewMaj}.} Although skew tableaux are not an instance of Kirillov and Reshetikhin's theorem, we can prove the unimodality of the formula based on the previously proven cocentricity of $f_{(n-k,k),i}$ for all $k$ and fixed $n$.

The formula for skew tableaux can readily be written as a telescoping sum.

\begin{multline*}
 q^{i^2} \left( \left[ { {n-j} \atop {i} } \right]_q \left[ { {k} \atop {i} } \right]_q - \left[ { {n+1} \atop {i} } \right]_q \left[ { {k-j-1} \atop {i} } \right]_q \right) = \\
  q^{i^2} \left( \left[ { {n} \atop {i} } \right]_q \left[ { {k-j} \atop {i} } \right]_q - \left[ { {n+1} \atop {i} } \right]_q \left[ { {k-j-1} \atop {i} } \right]_q \right) \\
  +  q^{i^2} \left( \left[ { {n-1} \atop {i} } \right]_q \left[ { {k-j+1} \atop {i} } \right]_q - \left[ { {n} \atop {i} } \right]_q \left[ { {k-j} \atop {i} } \right]_q \right) \\
  + \dots + q^{i^2} \left( \left[ { {n-j} \atop {i} } \right]_q \left[ { {k} \atop {i} } \right]_q - \left[ { {n-j+1} \atop {i} } \right]_q \left[ { {k-1} \atop {i} } \right]_q \right).
\end{multline*}

Each term on the right hand side is an instance of $f_{(n-j+s,k-s),i}$ for various $s$.  Since all of these are the distributions of the major index over standard Young tableaux with size $n-j+s+k-s = n-j+k$ and exactly $i$ descents for their various $s$, and hence by our earlier argument for Conjecture \ref{SagConj} are cocentric at degree $(n-j+k)i/2$, it follows that their sum is unimodal.  \hfill $\Box$

\phantom{.}

Before proving Theorem \ref{BigSchur}, we note a special case of interest in prior literature, namely the case of three-rowed tableaux which have the maximum number of descents.  For these, we have the following formula.

\begin{corollary}\label{ThreeRowFull} For $n \geq j$, $j \geq k$, $k \geq 0$, we have $$f_{(n,j,k),j+k} = q^{j^2+jk+k^2} \left( \left[ { {n} \atop {j} } \right]_q \left[ { {n} \atop {k} } \right]_q - q^{j-k+1} \left[ { {n} \atop {j+1} } \right]_q \left[ { {n} \atop {k-1} } \right]_q \right).$$
\end{corollary}

This family of formulas was studied by Butler, who showed that the coefficients were nonnegative \cite[Proposition 3.1]{Butler} by describing an injection on the inversion statistic on certain words; Stanton communicated in \cite{Butler} a proof via the Jacobi-Trudi identity, which also establishes the unimodality of the polynomial, by interpretation as a $q$-shift of the Schur function principal specialization $$q^{\binom{j+k}{2}}s_{(2^k,1^{j-k})} (1,q,\dots,q^{n-1}).$$ Unimodality also follows from Kirillov and Reshetikhin.

We now prove the more general Theorem \ref{BigSchur} of which the above result is the corollary.

\phantom{.}

\noindent \textbf{Proof of Theorem \ref{BigSchur}.} We begin by establishing boundary conditions for a recurrence.

For $\lambda = (\lambda_1,0,0,0,\dots,0)$, $\lambda_1 \geq 0$, we have one standard Young tableau with major index zero, so $f_{\lambda,0} = 1$.  The claim of the theorem is that

$$f_{\lambda,0} = q^{\binom{1}{2}} det \left( \left[ { {\lambda_1} \atop {0-i+j} } \right]_q q^{\binom{0-i+j}{2}} \right)_{2 \leq i,j \leq r} = 1$$ \noindent and so the theorem holds in this case.

The theorem holds if instead of a partition $\lambda$ we have $\beta = (\beta_1,\beta_2,\dots)$ in which some nonempty set of $\beta_i$ are equal to $\beta_{i+1} - 1$, but otherwise $\beta_i \geq \beta_{i+1}$.  If $\beta_1 = \beta_2 - 1$ then the first row of the determinant is zero.  In any other case where $\beta_i = \beta_{i+1} - 1$, the determinant is zero since two or more rows are equal.  In both cases this is correct since we do not consider tableaux of non-partition shape to exist.

If $\lambda = (\lambda_1,\dots,\lambda_r,0,0,\dots)$ with $\lambda_r = 0$, then the final row of the determinant is $(0 \dots 0 \, 1)$ and expansion of the determinant across this row allows us to consider the case of $\lambda$ having $r-1$ nonnegative parts.

Thus, considering the theorem for case $\lambda = (\lambda_1,\dots,\lambda_r,0,0,\dots)$, $\lambda_i > 0$, we can assume for the sake of induction that the statement of the theorem holds for any vector $\beta$ in which $\lambda_i - 1 \leq \beta_i \leq \lambda_i$.  This suffices for the induction we will require.

Say $\lambda = (\lambda_1,\dots,\lambda_r,0,0,\dots) \vdash n$, $\lambda_i > 0$.  Suppose $S \subseteq \{ 1, 2, \dots, r \}$.  Let $\chi$ be the membership function $\chi(i) = 1$ if $i \in S$, $\chi(i) = 0$ if $i \not\in S$.  Denote $$\lambda^{\downarrow S } := (\lambda_1 - \chi(1), \lambda_2 - \chi(2), \dots, \lambda_r - \chi(r),0,0,\dots ).$$

Consider any standard Young tableau of shape $\lambda$ having the maximum number $n - \lambda_1$ of descents.  If the box containing $n$ is in the first row, then it may be removed to leave the partition $(\lambda_1 - 1,\lambda_2,\dots)$ which still has the same number of descents and the same major index.

If the box containing $n$ is in a row lower than the first, then it is the bottom of some string of boxes $n, n-1, \dots, n-\vert S \vert$, where $S$ is some set of rows in the tableau containing these boxes, each of which are the bottom of a descent from the previous box in a higher row, terminating after $\vert S \vert$ steps at the last box in the first row.  For example, consider the tableau

$$ \young(1257,368:,4:::,9:::).$$

\noindent Here $S = \{2, 4\}$.

Removing this string of boxes, including the box in the first row, removes $\vert S \vert$ descents and reduces the major index of the tableau by $\vert S \vert n - \binom{\vert S \vert + 1}{2}$.  The remaining tableau has the new maximum number of descents for the remaining shape, and may have any valid collection of rows extended.  In the example tableau, after removal we obtain the tableau

$$ \young(125,36:,4::).$$

This yields the recurrence

$$f_{\lambda,n-\lambda_1} = f_{\lambda^{\downarrow \{ 1 \}},n-\lambda_1} + \sum_{{S \subseteq \{ 2,\dots, r \}} \atop {S \neq \emptyset}} q^{\vert S \vert n - \binom{\vert S \vert + 1}{2}} f_{\lambda^{\downarrow (S \bigcup \{1 \})},n-\lambda_1-\vert S \vert}.$$

We observe that the first term can be interpreted as $S = \emptyset$ without loss of correctness, and so we may simply write

$$f_{\lambda,n-\lambda_1} = \sum_{S \subseteq \{ 2,\dots, r \}} q^{\vert S \vert n - \binom{\vert S \vert + 1}{2}} f_{\lambda^{\downarrow (S \bigcup \{1 \})},n-\lambda_1-\vert S \vert}.$$

By induction, the claim of the theorem becomes

\begin{multline*} q^{\binom{n-\lambda_1+1}{2}} det \left( \left[ { {\lambda_1} \atop {\lambda_i-i+j} } \right]_q q^{\binom{\lambda_i-i+j}{2}} \right)_{2 \leq i,j \leq r} = \\ \sum_{S \subseteq \{ 2, \dots, r \} } q^{\vert S \vert n - \binom{\vert S \vert + 1}{2}} q^{\binom{n-\lambda_1-\vert S \vert + 1}{2}} det \left( \left[ { {\lambda_1 - 1} \atop {\lambda_i - \chi(i) -i+j} } \right]_q q^{\binom{\lambda_i - \chi(i) - i+j}{2}} \right)_{2 \leq i, j \leq r}.
\end{multline*}

Consider terms in which $r \not\in S$, and pair these with terms in which the rows reduced are $S \bigcup \{ r \}$.  Let $\delta_{i,r}$ be the Dirac delta returning 1 if $i=r$ and 0 otherwise.  We have the identity

\begin{multline*}
q^{\vert S \vert n-\binom{\vert S \vert +1}{2}+\binom{n-\lambda_1-\vert S \vert +1}{2}} det \left( \left[ { {\lambda_1-1} \atop {\lambda_i-\chi(i)-i+j} } \right]_q q^{\binom{\lambda_i - \chi(i)-i+j}{2}} \right)_{2 \leq i, j \leq r} \\
+ q^{(\vert S \vert +1)n-\binom{\vert S \vert +2}{2}+\binom{n-\lambda_1-\vert S \vert }{2}} det \left( \left[ { {\lambda_1-1} \atop {\lambda_i-\chi(i)-i+j-\delta_{i,r}} } \right]_q q^{\binom{\lambda_i - \chi(i)-i+j-\delta_{i,r}}{2}} \right)_{2 \leq i, j \leq r} \\
= q^{\vert S \vert n-\binom{\vert S \vert +1}{2}+\binom{n-\lambda_1-\vert S \vert +1}{2}} \left[  det \left( \left[ { {\lambda_1-1} \atop {\lambda_i-\chi(i)-i+j} } \right]_q q^{\binom{\lambda_i - \chi(i)-i+j}{2}} \right)_{2 \leq i, j \leq r} \right. \\
\left. + q^{\lambda_1-1} det \left( \left[ { {\lambda_1-1} \atop {\lambda_i-\chi(i)-i+j-\delta_{i,r}} } \right]_q q^{\binom{\lambda_i - \chi(i)-i+j-\delta_{i,r}}{2}} \right)_{2 \leq i, j \leq r} \right]
\end{multline*}

Recall that a determinant of an $n \times n$ matrix is a sum over permutations of length $n$.  In the two determinants of the latter line, consider corresponding terms that are indexed by the permutation $\sigma = (\sigma_2,\dots,\sigma_r) \in {\mathfrak{S}}_{r-1}$.  Denote $\beta_i = \lambda_i - \chi(i) - i + \sigma_i$.  The terms in the summations corresponding to this permutation, ignoring leading powers of $q$, are then

\begin{multline*}
\prod_{i=2}^r \left[ { {\lambda_1 - 1} \atop {\beta_i}} \right]_q q^{\binom{\beta_i}{2}} + q^{\lambda_1-1} \prod_{i=2}^r \left[ { {\lambda_1 - 1} \atop {\beta_i - \delta_{i,r}}} \right]_q q^{\binom{\beta_i - \delta_{i,r}}{2}} \\
= \left( \prod_{i=2}^r \left[ { {\lambda_1 - 1} \atop {\beta_i}} \right]_q q^{\binom{\beta_i}{2}} \right) q^{\binom{\beta_r}{2}} \left( \left[ { {\lambda_1 - 1} \atop {\beta_r} } \right]_q + q^{\lambda_1 - \beta_r} \left[ { {\lambda_1 - 1} \atop {\beta_r - 1} } \right]_q \right) \\
= \left( \prod_{i=2}^r \left[ { {\lambda_1 - 1} \atop {\beta_i}} \right]_q q^{\binom{\beta_i}{2}} \right) q^{\binom{\beta_r}{2}} \left[ { {\lambda_1} \atop {\beta_r} } \right]_q
\end{multline*}

\noindent by the identity $\left[ {M \atop N} \right]_q = \left[ {{M-1} \atop N} \right]_q + q^{M-N} \left[ { {M-1} \atop {N-1} } \right]_q$.

The claim of the theorem now becomes 

\begin{multline*} q^{\binom{n-\lambda_1+1}{2}} det \left( \left[ { {\lambda_1} \atop {\lambda_i-i+j} } \right]_q q^{\binom{\lambda_i-i+j}{2}} \right)_{2 \leq i,j \leq r} = \\ \sum_{S \subseteq \{ 2, \dots, r-1 \} } q^{\vert S \vert n - \binom{\vert S \vert + 1}{2}} q^{\binom{n-\lambda_1-\vert S \vert + 1}{2}} det \left( \left[ { {\lambda_1 - 1 + \delta_{i,r} } \atop {\lambda_i - \chi(i) -i+j} } \right]_q q^{\binom{\lambda_i - \chi(i) - i+j}{2}} \right)_{2 \leq i, j \leq r}.
\end{multline*}

In other words we have the same sum except that $\vert S \vert$ never contains $r$, and the $r$ row of the determinant has $\lambda_1$ as the upper entry of the $q$-binomial coefficients instead of $\lambda_1-1$.

The same matching and $q$-binomial identity can be employed repeatedly with index $r-1$, $r-2$, etc., up to index 2.  When only $S = \emptyset $ remains, all terms on the right hand side match the corresponding term on the left, and the claim of the theorem is verified. \hfill $\Box$

\phantom{.}

\subsection{Three-rowed tableaux}

We prove several families of formulas for three-rowed tableaux.  In most of these cases not covered earlier, it is apparently \emph{not} the case that the formulas are instances of principal specializations.  Exhaustive computer calculations confirm that, for instance, $f_{(3,3,3),3}$ and $f_{(4,3,3),3}$ are not $q$-power shifts of $$det \left( \left[ { {\lambda_1} \atop {\lambda_i-i+j} } \right]_q q^{\binom{\lambda_i-i+j}{2}} \right)_{2 \leq i,j \leq r}$$ \noindent for any partition $\lambda$ of 3 or 4 parts of size less than 15, and it seems improbable that they should be a specialization involving larger parts.

Kirillov and Reshetikhin's result does still confirm that all of the formulas proven in this subsection are unimodal.  In several cases, however, this is obvious from the construction.

The typical method for a theorem proved in this section is to establish a suitable recurrence and boundary conditions, then induct.  Some of the proofs in this section omit repetitive details.

\phantom{.}

\noindent \textbf{Proof of Theorem \ref{MK2Thm}.}  For a given $n$, $k$, and $i$, suppose for the sake of induction that the formula holds for smaller values of the parameters.  The formula yields 0 for $i \leq 1$, $n=k-1$, or $k=1$, as desired.

For $k=2$, $i=2$ the claim of the formula is that $f_{(n,2,2),2} = q^6 \left[ {n \atop 2} \right]_q$.  Tableaux of such characteristics are uniquely identified by descents occurring at distinct positions $p_1$ and $p_2$ with $2 \leq p_1 \leq p_2 - 2 \leq n$.  Such pairs with $p_1 + p_2 = m+6$ are in correspondence with partitions of $m$ counted by the $q$-binomial coefficient $\left[ {n \atop 2} \right]_q$ and so the theorem holds for the smallest nonzero case.

In a given tableau, observe the position of the final box $n+k+2$.  Denote the distribution of major index over tableaux of shape $\lambda$ with exactly $i$ descents in which the box containing $\vert \lambda \vert$ is in row $j$ by $f_{\lambda,i}^{(j)}$.

We have the following recurrences:

\begin{align*}
f_{(n,k,2),i}^{(1)} &= f_{(n-1,k,2),i} = f_{(n-1,k,2),i}^{(1)} + f_{(n-1,k,2),i}^{(2)} + f_{(n-1,k,2),i}^{(3)} \\
f_{(n,k,2),i}^{(2)} &= q^{n+k+1} f_{(n,k-1,2),i-1}^{(1)} + f_{(n,k-1,2),i}^{(2)} + f_{(n,k-1,2),i}^{(3)} \, . \\
f_{(n,k,2),i}^{(3)} &= q^{n+k+1} f_{(n,k,1),i-1}^{(1)} + q^{n+k+1} f_{(n,k,1),i-1}^{(2)} + f_{(n,k,1),i}^{(3)}
\end{align*}

We further observe that $f_{(n,k,1),i}^{(3)} = q^{n+k} f_{(n,k),i-1}$ and that $$q^{n+k+1} f_{(n,k,1),i-1}^{(1)} + q^{n+k+1} f_{(n,k,1),i-1}^{(2)} = q^{n+k+1} f_{(n,k,1),i-1} - q^{n+k+1} f_{(n,k,1),i-1}^{(3)}.$$

This yields

$$f_{(n,k,2),i}^{(3)} = q^{n+k+1} f_{(n,k,1),i-1} - q^{2n+2k+1} f_{(n,k),i-2} + q^{n+k} f_{(n,k),i-1}.$$

Similarly,

\begin{multline*}f_{(n,k,2),i}^{(2)} = q^{n+k+1} f_{(n,k-1,2),i-1}^{(1)} + f_{(n,k-1,2),i}^{(2)} + f_{(n,k-1,2),i}^{(3)} \\
= q^{n+k+1} f_{(n-1,k-1,2),i-1} + f_{(n,k-1,2),i}^{(2)} + f_{(n,k-1,2),i}^{(3)} \\
= q^{n+k+1} f_{(n-1,k-1,2),i-1} + f_{(n,k-1,2),i} - f_{(n,k-1,2),i}^{(1)} \\
= q^{n+k+1} f_{(n-1,k-1,2),i-1} + f_{(n,k-1,2),i} - f_{(n-1,k-1,2),i} \\
\end{multline*}

Thus

\begin{multline}\label{nk2rec}
f_{(n,k,2),i} = f_{(n,k,2),i}^{(1)} + f_{(n,k,2),i}^{(2)} + f_{(n,k,2),i}^{(3)} \\
= f_{(n-1,k,2),i} + q^{n+k+1} f_{(n-1,k-1,2),i-1} + f_{(n,k-1,2),i} - f_{(n-1,k-1,2),i} \\ +  q^{n+k+1} f_{(n,k,1),i-1} - q^{2n+2k+1} f_{(n,k),i-2} + q^{n+k} f_{(n,k),i-1} .
\end{multline}

Into equation \ref{nk2rec} we now substitute the conjectured formula on the left, and the known formulas for $f_{(n,k,1),i}$ and $f_{(n,k),i}$ as well as the inductively assumed formula for $f_{(n,k,2),i}$ on the right.  We obtain that we wish to verify the identity

\begin{multline*}
q^{k+i^2-3i+6} \frac{(1-q^{n-k+1})(1-q^{k-1})(1-q^n)}{(1-q^{i-1})(1-q)(1-q^2)} \left[ { {n+1} \atop {i-2} } \right]_q \left[ { {k} \atop {i-2} } \right]_q \\
= q^{k+i^2-3i+6} \frac{(1-q^{n-k})(1-q^{k-1})(1-q^{n-1})}{(1-q^{i-1})(1-q)(1-q^2)} \left[ { {n} \atop {i-2} } \right]_q \left[ { {k} \atop {i-2} } \right]_q \\
+ q^{n+2k+i^2-5i+10} \frac{(1-q^{n-k+1})(1-q^{k-2})(1-q^{n-1})}{(1-q^{i-2})(1-q)(1-q^2)} \left[ { {n} \atop {i-3} } \right]_q \left[ { {k-1} \atop {i-3} } \right]_q \\
+ q^{k+i^2-3i+5} \frac{(1-q^{n-k+2})(1-q^{k-2})(1-q^n)}{(1-q^{i-1})(1-q)(1-q^2)} \left[ { {n+1} \atop {i-2} } \right]_q \left[ { {k-1} \atop {i-2} } \right]_q \\
- q^{k+i^2-3i+5} \frac{(1-q^{n-k+1})(1-q^{k-2})(1-q^{n-1})}{(1-q^{i-1})(1-q)(1-q^2)} \left[ { {n} \atop {i-2} } \right]_q \left[ { {k-1} \atop {i-2} } \right]_q \\
+ q^{n+2k+i^2-4i+6} \frac{(1-q^{n-k+1})(1-q^{i-2})}{(1-q^{i-1})(1-q)}\left[{k \atop {i-2}} \right]_q \left[ {{n+1} \atop {i-2}} \right]_q\\
- q^{2n+2k+1+i^2} \left( \left[ { {n} \atop {i-2} } \right]_q \left[ { {k} \atop {i-2} } \right]_q - \left[ { {n+1} \atop {i-2} } \right]_q \left[ { {k-1} \atop {i-2} } \right]_q  \right) \\
+ q^{n+k+i^2} \left( \left[ { {n} \atop {i-1} } \right]_q \left[ { {k} \atop {i-1} } \right]_q - \left[ { {n+1} \atop {i-1} } \right]_q \left[ { {k-1} \atop {i-1} } \right]_q  \right).
\end{multline*}

To verify this identity, multiply through both sides of the equation by

$$(1-q)(1-q^2)(1-q^{i-1})(1-q^{i-2})\frac{(q)_{i-1}(q)_{n-i+3}(q)_{k-i+2}(q)_{i-1}}{q^{k+i^2-5i} (q)_n (q)_{k-1}}.$$

We obtain that we wish to verify the polynomial identity

\begin{multline*}
q^{2i+6}(1-q^{i-2})(1-q^{n-k+1})(1-q^{k-1})(1-q^n)(1-q^{n+1})(1-q^{i-1})(1-q^k)(1-q^{i-1}) \\
= q^{2i+6}(1-q^{i-2})(1-q^{n-k})(1-q^{k-1})(1-q^{n-1})(1-q^{n-i+3})(1-q^{i-1})(1-q^k)(1-q^{i-1}) \\
+ q^{n+k+10}(1-q^{i-1})(1-q^{n-k+1})(1-q^{k-2})(1-q^{n-1})(1-q^{i-2})(1-q^{i-1})(1-q^{i-2})(1-q^{i-1}) \\
+ q^{2i+5}(1-q^{i-2})(1-q^{n-k+2})(1-q^{k-2})(1-q^n)(1-q^{n+1})(1-q^{i-1})(1-q^{k-i+2})(1-q^{i-1}) \\
- q^{2i+5}(1-q^{i-2})(1-q^{n-k+1})(1-q^{k-2})(1-q^{n-1})(1-q^{n-i+3})(1-q^{i-1})(1-q^{k-i+2})(1-q^{i-1}) \\
+ q^{n+k+i+6}(1-q^2)(1-q^{i-2})^2(1-q^{n-k+1})(1-q^{n+1})(1-q^{i-1})(1-q^k)(1-q^{i-1}) \\
- q^{2n+k+1+5i}(1-q^{i-2})(1-q^{i-1})(1-q)(1-q^2)\left((1-q^{n-i+3})(1-q^{i-1})(1-q^k)(1-q^{i-1}) \right. \\ \left. - (1-q^{n+1})(1-q^{i-1})(1-q^{k-i+2})(1-q^{i-1}) \right) \\
+ q^{n+5i}(1-q^{i-2})(1-q^{i-1})(1-q)(1-q^2)\left((1-q^{n-i+2})(1-q^{n-i+3})(1-q^k)(1-q^{k-i+2}) \right. \\ \left. - (1-q^{n+1})(1-q^{n-i+3})(1-q^{k-i+1})(1-q^{k-i+2}) \right) .
\end{multline*}

Expansion and cancellation (with a symbolic algebra package to ease the calculations) verifies the identity, and the theorem is proved. \hfill $\Box$

We have the following results for tableaux of shape $(n,3,3)$.

\begin{theorem}\label{N33}
\begin{align*}
f_{(n,3,3),2} &= q^9 \left[ {{n-1} \atop 2} \right]_q \\
f_{(n,3,3),3} &= q^{11} \left[ { {n-1} \atop {2} } \right]_q \left[ { {n+3} \atop {1} } \right]_q + q^{12} \left[ { {n} \atop {3} } \right]_q \left[ {4 \atop 1} \right]_q \\
f_{(n,3,3),4} &= q^{15} \left[ { {n-1} \atop {2} } \right]_q \left[ { {n+2} \atop {2} } \right]_q + q^{15} \left[ { {n+1} \atop {4} } \right]_q \left[ { {5} \atop {2} } \right]_q \\
f_{(n,3,3),5} &= q^{21} \left[ { {n-1} \atop {2} } \right]_q \left[ { {n+1} \atop {3} } \right]_q + q^{20} \left[ { {n+1} \atop {4} } \right]_q \left[ { {n+3} \atop {1} } \right]_q  \\
f_{(n,3,3),6} &= q^{27} \frac{(1-q^{n-2})(1-q^{n-1})^2 (1-q^n)^2(1-q^{n+1})}{(1-q)(1-q^2)^2(1-q^3)^2 (1-q^4)} \\ 
 &= q^{27} \left[ { {n} \atop {3} } \right]_q \left[ { {n} \atop {3} } \right]_q - q^{28} \left[ { {n} \atop {4} } \right]_q \left[ { {n} \atop {2} } \right]_q
\end{align*}
\end{theorem}

\begin{proof}

The case $f_{(n,3,3),2}$ is established with an argument similar to the boundary case of the previous theorem; tableaux of these characteristics are uniquely identified by a partition into two parts of size at least 3 and at most $n+3$, with the second part differing from the first by at least 3.

The case $f_{(n,3,3),6}$ is handled by Theorem \ref{BigSchur}.

The remaining cases are handled by induction and recurrence.  We observe that the box containing $n+6$ must be either in the first row, or the third.  When removing a box from the third row we invoke the formulas given in previous theorems for shapes $\lambda = (n,3,2)$, $(n,3,1)$, or $(n,3)$. The argument is otherwise similar to the proof of Theorem \ref{MK2Thm}.

\end{proof}

\noindent \textbf{Remark:} We note that the unimodality of the distributions above is more obvious than for previous formulas; the first case is trivial, the final case was established with Theorem \ref{BigSchur}, and in the intermediate cases the central degrees of both terms are equal from inspection.

\section{Future work}\label{FutureSec}

The present work leaves many lines of investigation open with great possibility of fruitful exploration.

\subsection{General three-rowed tableaux.} The formulas proven above for small cases of three-rowed tableaux suggest that a general formula for these should not be too difficult to obtain.

Once a distribution is conjectured for a specific family of tableau shapes and descents, and if formulae are known for all partition shapes contained within the desired shape, then a relatively direct route to a proof is induction and recurrence using the method of Theorem \ref{BigSchur} followed by verification of a polynomial identity.  This can probably be carried out for general three part partitions, for which one general recurrence is, letting $\lambda = (\lambda_1,\lambda_2,\lambda_3) \vdash n$: 

\begin{multline*}f_{\lambda,i} = f_{(\lambda_1-1,\lambda_2,\lambda_3),i} + q^{n-1}f_{(\lambda_1-1,\lambda_2-1,\lambda_3),i-1} + f_{(\lambda_1,\lambda_2-1,\lambda_3),i}-f_{(\lambda_1-1,\lambda_2-1,\lambda_3),i} \\
+ \sum_{\ell=1}^{\lambda_3} \sum_{j=1}^{\ell} \sum_{{S = \{ s_1,\dots , s_j \}} \atop {S \subseteq \{1,2,\dots , \ell \}}} (-1)^{\vert S \vert -1} q^{\vert S \vert n -  \sum s_k} f_{(\lambda_1,\lambda_2,\lambda_3 -\ell),i-j} .
\end{multline*}

Generally useful terminal points for this recurrence might be $f_{(n,n,k),i}$ and $f_{(n,k,k),i}$, examples of which are the following.

\begin{conjecture} \begin{align*}
f_{(n,n,3),3} &= q^{n+8} \left[ {{n+2} \atop 1} \right]_q \left[ { {n} \atop {3} } \right]_q \left[{2 \atop 1} \right]_q \\
f_{(n,4,4),3} &= q^{14} \left[ {{n-2} \atop 2} \right]_q \left[ {n \atop 1} \right]_q \left[ {6 \atop 1} \right]_q - q^{17} \frac{(1-q^4)(1-q^{n-3})^2 (1-q^{n-2})}{(1-q)^2 (1-q^2)^2}
\end{align*}
\end{conjecture}

Although the formulas for three-rowed tableaux proven to date are not apparently instances of Jacobi-Trudi determinants, it is conceivable that they are simple $q$-linear combinations of these; if so, this would be a useful property to establish.

\subsection{General partition shapes; refinement of Stanley's formula.} A second general recurrence can be established which might be extensible to partitions with a greater number of parts.  Suppose $\lambda \vdash n$ and that the first row ends with box $k$. Consider the boxes $k+1$ through $n$ as a skew partition $\mu$.  Let $\lambda / (\mu+1)$ denote $\lambda$ with these boxes $k$ through $n$ missing; suppose $\mu$ to contain $j$ descents.  Then we have the recurrence, summing over all valid $\mu$,

$$f_{\lambda,i} = \sum_{\mu} \sum_{j=0}^{i-1} f_{\lambda / (\mu+1),i-j-1} \cdot q^{(j+1)(n-\vert \mu \vert)} f_{\mu,j}.$$

A challenge to this line of approach is that many candidates exist for potential forms of formulas, since products of $q$-binomial coefficients and their shifts are vastly more populous than any potential basis for the spaces of polynomials being considered.  Numerical calculations can offer many suggestions; an investigator's mathematical intuition must suggest which families persist over all values of the parameters of interest.  For instance, it is entirely possible that a purely $q$-binomial expression for $f_{(n,4,4),3}$ exists.

As example conjectures, the following families appear to be natural expressions for the related partition shapes; they do not appear to be naturally products of $q$-binomial coefficients.  For larger numbers of parameters even conjecturing an appropriate family becomes a more difficult task.

\begin{conjecture}
\begin{align*}
f_{(n,k,3),2} &= q^{k+6} \frac{(1-q^{n-1})(1-q^{n-k+1})(1-q^{k-2})}{(1-q)^2(1-q^2)} \\
 f_{(n,4,3),3} &= q^{12} \frac{(1-q^5)(1-q^{n-3})(1-q^n)^2}{(1-q)^3(1-q^2)} + q^{13} \left[ { {n-2} \atop 2} \right]_q \left[ { {n+4} \atop 1} \right]_q \\
 f_{(n,5,3),3} &= q^{13} \frac{(1-q^5)(1-q^{n-3})(1-q^{n-1})(1-q^{n+2})}{(1-q)^3(1-q^2)} \\
  &+ q^{14} \frac{(1-q^6)(1-q^{n-5})(1-q^{n-1})(1-q^{n+1})}{(1-q)^3(1-q^2)} \\
\end{align*}
\end{conjecture}

Recall the Stanley $q$-analogue of the Frame-Robinson-Thrall formula, equation \ref{StanFRT}, which gives $\sum_{\tau \in SYT(\lambda)} q^{maj(\tau)}$.  An ambitious goal for this line of investigation would be a short formula for the complete refinement $$\sum_{{\tau \in SYT(\lambda)} \atop {des(\tau) = i}} q^{maj(\tau)}.$$

If more formulas for the smaller cases could be established, a sufficient library of these could lead investigators to a correct form for such a refinement which could then be proved by very general recurrence, or other means.  The skew-tableaux approach has potential, although the complexity of the Naruse hooklength formula for the number of tableaux of skew shape (\cite{Naru}, see \cite{MPP} for thorough explication and context) suggests obstacles.

\subsection{Extension of the Kirillov/Reshetikhin unimodality result.} As the formulas given in this paper for skew shapes are not covered by the result of Kirillov and Reshetikhin on unimodality, and several formulas found are not instances of the Jacobi-Trudi identities, it is foreseeable that it would be useful to extend Kirillov and Reshetikhin's result to skew shapes.

The unimodality of the distribution of the major index for some skew tableaux, as proven in Theorem \ref{SkewMaj}, suggests that Kirillov and Reshetikhin's result on unimodality can be extended.  Although $f_{\lambda \setminus \mu,i}$ is not always symmetric when $\lambda$ has more than 3 parts, calculations do suggest the conjecture

\begin{conjecture} The polynomials $f_{\lambda \setminus \mu,i}$ are unimodal.
\end{conjecture}

Even if a counterexample is found, understanding the conditions under which the conjecture is true or false would be of significant interest.

\section{Acknowledgements}

Portions of this work were presented at the conference Combinatory Analysis 2018 at Pennsylvania State University, celebrating the 80th birthday of George Andrews.  The author thanks the organizers of that conference for the invitation to speak, and fellow attendee Dennis Stanton of the University of Minnesota for alerting him to the Jacobi-Trudi interpretations of his results, which led to Theorem \ref{BigSchur} and its corollary.

\end{document}